\documentclass{amsart}
\usepackage[mathscr]{eucal}
\usepackage{color}
\usepackage{xypic}
\usepackage{amsfonts}   
\usepackage{amsmath}
\usepackage{amsthm}
\usepackage{amssymb}
\usepackage{latexsym}
\usepackage[english]{babel}
\usepackage[utf8]{inputenc}
\def\neweq{\setcounter{theorem}{0}}

\let\nc\newcommand
\let\rnc\renewcommand

\nc{\la}{\label}

\def\bg{\begin}
\newtheorem*{conjecture}{Conjecture}

\newtheorem{theorem}{Theorem}[section]
\newtheorem{definition}[theorem]{Definition}
\newtheorem{corollary}[theorem]{Corollary}
\newtheorem{lemma}[theorem]{Lemma}
\newtheorem{proposition}[theorem]{Proposition}

\newtheorem{remark}{Remark}
\newtheorem*{problem}{Problem}
\def\k{\mathsf k}
\def\C{\mathbb C}
\def\Z{\mathbb Z}
\def\A{\mathcal Q}
\def\D{\mathbb D}
\def\F{\mathbb F}
\def\L{\mathcal L}
\def\g{\mathfrak g}

\newcommand{\eq}{equation}
\newcommand{\Mod}{{\tt{Mod}}}
\newcommand{\Mat}{{\tt{Mat}}}
\newcommand{\Alg}{{\tt{Alg}}}
\newcommand{\HH}{{\rm{HH}}}
\newcommand{\Res}{{\rm{Res}}}
\newcommand{\Ann}{{\rm{Ann}}}
\newcommand{\Frac}{{\rm{Frac}}}
\newcommand{\eps}{\varepsilon}
\newcommand{\Tor}{{\rm{Tor}}}
\newcommand{\pr}{{\tt can\,}\,}
\newcommand{\ol} {\mbox{\rm{$1$-length}}}
\newcommand{\len} {\mbox{\rm{length}}}
\newcommand{\Spec}{{\rm{Spec}}}
\newcommand{\Span}{{\rm{span}}}
\newcommand{\Pic}{{\rm{Pic}_{\c}}}
\newcommand{\Aut}{{\rm{Aut}_{\c}}}
\newcommand{\Autk}{{\rm{Aut}}}
\newcommand{\Autg}{{\rm{Aut}_{\Gamma}}}
\newcommand{\id}{{\rm{Id}}}
\newcommand{\Der}{{\rm{Der}}}
\newcommand{\Div}{{\rm{Div}}}
\newcommand{\rk}{{\rm{rk}}}
\newcommand{\Tr}{{\rm{Tr}}}
\newcommand{\tr}{{\rm{tr}}}
\newcommand{\Ker}{{\rm{Ker}}}
\newcommand{\Inn}{{\rm{Inn}}}
\newcommand{\diag}{{\rm{diag}}}
\newcommand{\Coker}{{\rm{Coker}}}
\newcommand{\ev}{{\rm{ev}}}
\newcommand{\im}{{\rm{Im}}}
\newcommand{\balpha}{{\boldsymbol{\alpha}}}
\newcommand{\bmu}{{\boldsymbol{\mu}}}
\newcommand{\bn}{{\boldsymbol{n}}}
\newcommand{\bk}{{\boldsymbol{k}}}
\newcommand{\bU}{{\boldsymbol{U}}}
\newcommand{\bL}{{\boldsymbol{L}}}
\newcommand{\bM}{{\boldsymbol{M}}}
\newcommand{\be}{{\boldsymbol{\rm e}}}
\newcommand{\Ui}{U_{\infty}}
\newcommand{\blambda}{{\boldsymbol{\lambda}}}
\newcommand{\bgg}{{\boldsymbol{g}}}
\newcommand{\Vi}{V_{\infty}}
\newcommand{\vi}{v_{\infty}}
\newcommand{\ei}{e_{\infty}}
\newcommand{\lambdai}{\lambda_{\infty}}
\newcommand{\supp}{{\rm{supp}}}
\newcommand{\into}{\,\,\hookrightarrow\,\,}
\newcommand{\too}{\,\,\longrightarrow\,\,}
\newcommand{\onto}{\,\,\twoheadrightarrow\,\,}

\def\k{\mathsf k}
\def\C{\mathbb C}
\def\Z{\mathbb Z}
\def\A{\mathbb A}
\def\D{\mathbb D}
\def\F{\mathbb F}
\def\L{\mathcal L}
\def\g{\mathfrak g}

\newtheorem{corol}[theorem]{Corollary}

\newenvironment{dedication}
  {
   \itshape             
   \raggedleft          
  }

\AtEndDocument{\bigskip{\footnotesize%
 
  \addvspace{\medskipamount}
  H.\, L. \, Mariano \textsc{Instituto de Matematica e Estatistica, Universidade de S\~{a}o Paulo, \\ S\~{a}o Paulo, Brasil} \par
  \textit{E-mail address}: \texttt{hugomar@ime.usp.br}  \par 
  J.\,Schwarz \textsc{International Center for Mathematics \\ Shenzhen, P. R. China} \par
  \textit{E-mail address}: \texttt{jfschwarz.0791@gmail.com}  \par
   \addvspace{\medskipamount}
  
}}

\begin{document}

\title[Gelfand-Kirillov Conjecture as a first-order formula]{Gelfand-Kirillov Conjecture as a first-order formula}

\author{Hugo Luiz Mariano, Jo\~ao Schwarz}

\begin{abstract}

Let $\Sigma$ be a (reduced) root system. Let $\k$ be an algebraically closed field of zero characteristic, and consider the corresponding semisimple Lie algebra $\mathfrak{g}_{\k, \Sigma}$. Then there is a first-order sentence $\phi_\Sigma$ in the language $\mathcal{L}=(1,0,+,*,-)$ of rings such that, for any algebraically closed field $\k$ of characteristic 0, the validity of the Gelfand-Kirillov Conjecture for $\mathfrak{g}_{\k, \Sigma}$ is equivalent to $ACF_0 \vdash \phi_\Sigma$.



\end{abstract}

\maketitle

\begin{dedication}
    Dedicated to the memory of the brazillian philosopher, logician and mathematician Newton da Costa.
\end{dedication}
\section{Introduction}
\let\thefootnote\relax\footnotetext{MSC 2020 Primary: 03C60; Secundary: 16S85, 16W22, 17B35}
\let\thefootnote\relax\footnotetext{Keywords: First-order characterization, noncommutative birational equivalence, Gelfand-Kirillov Conjecture}

Connections between Algebra and Logic are well known (cf. \cite{Cohn}, \cite{Malcev}). In the specific topic of Algebraic Geometry, this line of inquiry began with the work of Alfred Tarski on the decidability via quantifier elimination in the theory of algebraically closed fields, and have achieved a remarkable development through the years using more sophisticated methods of Model Theory in Algebraic Geometry, such as the work of Ax, Kochen and Ershov on Artin's Conjecture, or the celebrated proof of Mordell-Lang Conjecture by Hrushovski (cf. \cite{Hodges}, \cite{Bouscaren}, \cite{HV}, \cite{MMP}).

Let's fix some conventions. All our rings and fields will be algebras over a base field $\k$.

One of the main problems of algebraic geometry is the birational classification of varieties (\cite{Hartshorne}). In case $\k$ is algebraically closed and we work in the category of affine irreducible varieties the situation is rather simple: given two finitely generated domains $A$ and $B$ and corresponding varieties $X = \operatorname{Spec} \, A, \, Y=\operatorname{Spec} \, B$, they are birationally equivalent if and only if $\operatorname{Frac} \, A = \operatorname{Frac} \, B$. In general, two varieties are birationally equivalent if their function fields are isomorphic fields \cite{Hartshorne}. For a state-of-the-art introduction to the subject, see \cite{Kollar}.

In the 1966 the study of birational geometry of noncommutative objects began. In his adress at the 1966 ICM in Moscow, A. A. Kirillov proposed to classify, up to birational equivalence, the enveloping algebras $U(\mathfrak{g})$ of finite dimensional algebraic Lie algebras $\g$ when $\k$ is algebraically closed of zero characteristic. This means to find canonical division rings such that every skew field $\operatorname{Frac} \, U(\mathfrak{g})$  of the enveloping algebras, which are an Ore domain$^1$ \footnote{$^1$ In general, it is not the case that a noncommutative domain can be embedded in a division ring, as shown by Malcev. Ore domains are an example of when a quotient division ring exists in a particularly nice form, cf. \cite[Chapter 4]{Lam}.}, is isomorphic to one of them.

The idea became mature in the groundbreaking paper \cite{Gelfand}$^2$\footnote{$2$ This is the same paper that introduced the now ubiquitous Gelfand-Kirillov dimension (see \cite[Chapter 8]{MR}), and also the important Gelfad-Kirillov transcendence degree}, where A. A. Kirillov and I. M. Gelfand formulated the celebrated Gelfand-Kirillov Conjecture. Before we formulate it, lets recall some definitions:

\begin{definition} \label{defWeyl}
The rank n Weyl algebra $A_n(\k)$ is the algebra given by generators $x_1, \ldots, x_n$ and $y_1, \ldots, y_n$ and relations $[x_i,x_j]=[y_i,y_j] = 0; [y_i,x_j]= \delta_{ij}$, $i,j=1,\ldots, n$. We denote by $A_{n,s}(\k)$ the algebra $A_n(\k(t_1, \ldots, t_s))$, for $n \geq 1, s \geq 0$. For the sake of notational simplicity, call $A_{0,s}(\k)=\k(t_1, \ldots, t_s)$. In characteristic 0, as is well known, the Weyl algebras are finitely generated simple Noetherian domains \cite{MR}. We denote by $\mathbb{D}_{n,s}(\k), \D_n(\k)$ the skew field of fractions of $A_{n,s}(\k), A_n(\k)$, respectively. These skew fields are called the Weyl fields.
\end{definition}

\begin{conjecture}
\emph{(Gelfand-Kirilov Conjecture)}: Consider the enveloping algebra $U(\mathfrak{g})$, $\mathfrak{g}$ a finite dimensional algebraic Lie algebra over $\k$ algebraically closed of zero characteristic. Its skew field of fractions, $\operatorname{Frac} \, U(\mathfrak{g})$, is of the the form $\mathbb{D}_{n,s}(\k)$, for some $n, s \geq 0$.
\end{conjecture}

The Conjecture was shown to be true for $\mathfrak{gl}_n$ and $\mathfrak{sl}_n$, as well for nilpotent Lie algebras, in \cite{Gelfand}. This later fact was generalized for solvable Lie algebras in \cite{McConnell}, \cite{Joseph}, \cite{Borho}. Other cases considered in \cite{Nghiem}  \cite{AOV2}, \cite{Ooms}; and certain modifications of it considered in \cite{Gelfand2} and \cite{Conze}. The Conjecture was very influential in the development of Lie theory (cf. \cite[Problem 3]{Dixmier}) and, recently, a $q$-analogue of the Conjecture is studied in quantum group theory (\cite[I.2.11, II.10.4]{Brown}). It also became a paradigmatic example on the study of skew field of fractions of many Ore domains (cf. \cite{Etingof}, \cite{AD1}, \cite{ADnew}). However, it was eventually shown that the Conjecture is false in general (\cite{AOV}). Regarding simple Lie algebras, it was known to be true in the seminal work of Gelfand and Kirillov for type $A$ simple Lie algebras. This question was revisited by Premet in \cite{Premet}, where using reduction module prime techniques, he showed the Conjecture to be false for types $B, D, E, F$. About the types $C, G$, nothing is known at this moment.

The purpose of this paper is to show that, surprisingly, given a (reduced) root system $\Sigma$ (cf. \cite[11.1]{Dixmier})  and \emph{any} algebraically closed field $\k$ with zero characteristic, the validity of the Gelfand-Kirillov Conjecture  for the finite dimensional Lie algebra $\mathfrak{g}_{\k, \Sigma}$ --- that is, the only semisimple Lie algebra over $\k$ whose associated root system is $\Sigma$ --- is  equivalent to the provability of a certain first-order sentence in the language of rings $\mathcal{L}(0,1,+,*,-)$ in the theory of algebraically closed fields of zero characteristic --- $ACF_0$ (cf. \cite{Hodges}).

Being more precise, we have:

\begin{theorem}\label{main1}
Given a root system $\Sigma$, there is a first-order sentence $\phi_\Sigma$ in the language of rings $\mathcal{L}(1,0,+,*,-)$ such that the below are equivalent:
\begin{enumerate}
\item
For some algebraically closed field $\k$ of zero characteristic, the Gelfand-Kirillov Conjecture holds for $\mathfrak{g}_{\k, \Sigma}$ .
\item
For all algebraically closed of zero characteristic $\k$ , the Gelfand-Kirillov Conjecture holds for  $\mathfrak{g}_{\k, \Sigma}$ .
\item
$ACF_0 \vdash \phi_\Sigma$.
\end{enumerate}

Moreover, $\phi_\Sigma$ is {\bf naturally} constructed as a existential closure of boolean combinations of atomic formulas in the language$^3$. \footnote{$^3$Obviously, since $ACF_0$ admits quantifier elimination, every sentence in the language of rings is equivalent to a boolean combination of atomic sentences.}
\end{theorem}

This theorem is surprising because there is no a priori reason for the Gelfand-Kirillov Conjecture to be expressed in a first-order formula. In what follows, $\A$ will denote the field of algebraic numbers.

To have a better understanding of what is happening, lets recall some ring theory. Let $\mathcal{R}=\mathcal{L}(1,0,+,*,-)$ be the first-order language of rings, and $\mathcal{R}^2$ the second-order language of rings.

Given a ring $E$, we want to define its Jacobson radical, denoted here by $J_E$, as a predicate: $J_E(r)$, where given $r \in E$, $J_E(r)$ means that $r \in J_E$. Our discussion of the Jacobson radical is classical; see, e. g.,  \cite{Lam}.

A first definition of $J_E$ might be: the intersection of the anihilators of all irreducible left $E$-modules. If this is to be the definition, we must work in ZFC: there will be a first-order formula in the language of ZFC, $\phi(x)$, containing $x$ as the only free variable, such that $J_E(r)$ if and only if

$ZFC \vdash \phi(r), \, r \in E$

Another description of the Jacobson radical $J_E$ is: the intersection of all left maximal ideals of $E$. In this case there is a formula $\phi_2(x) \in \mathcal{R}^2$, with only the variable $x$ free, such that $J_E(r)$ if and only if $E \models  \phi_2(r), \, r \in E$, in full second-order semantics.

Finally, as an equivalent characterization, we have: $r \in J_E$ if and only if for every $s \in E$, $1-sr$ is left-invertible in $E$. This is clearly expressible in first-order formula of the language $\mathcal{R}$.

Let $\Sigma$ be a root system and $\k$ an algebraically closed field of zero characteristic. We want to define the predicate $\mathcal{GK}(\k, \Sigma)$, that means that the Gelfand-Kirillov Conjecture is true for $\g_{\k,\Sigma}$. The initial definition is in ZFC.

If, for each $\k$, we had a formula $\theta_{\k, \Sigma}$ in the language $\mathcal{R}$ such that $\mathcal{GK}(\k, \Sigma)$ if and only if $\k \models \theta_{\k, \Sigma}$, we would already have a remarkable fact.

However, there is a  first-order formula $\theta_\Sigma$ in the language $\mathcal{R}$ such that $\mathcal{GK}(\k, \Sigma)$, for arbitrary $\k$ if and only if $\A \models \theta_\Sigma$. As $AFC_0$ is a complete theory, this holds if and only if $AFC_0 \vdash \theta_\Sigma$.

We remark also that the expression of a statement as a first-order sentence in $ACF_0$ is a very important question. One of the main applications of this idea is Lefchetz's Principle from algebraic geometry, since $ACF_0$ is a complete theory (\cite{Hodges}), in order to prove a statement for a variety over an algebraically closed field of zero characteristic, it suffices to show it for ${\mathsf k} = \mathbb{C}$, where transcendental methods are appliable. The Gelfand-Kirillov Conjecture is obviously in the realm of noncommutative algebraic geometry, but as a illustration of our result, the validity of the Gelfand-Kirillov Conjecture for simple Lie algebras of type C or G, where it is still open, over any algebraically closed field of zero characteristic, is equivalent to its validity over $\mathbb{C}$, where we can use complex Lie group techniques and other analytic and geometric theories. We remark, however, that since of course not every geometric aspect can be expressed in first-order logic, there is a need to be careful  (cf. \cite{Seidenberg})$^4$ \footnote{$^4$ However, the general feeling that Lefchetz's principle holds for any reasonably simple statement in algebraic geometry has been formalized in the work \cite{Eklof}.}

\section*{Acknowledgments}

J. S. was supported in part by FAPESP grant 2018/18146-5. The second author is grateful to M. Soares and S. Hoffmann Panucci.

\section{Proof of the Main Theorem}

Our method is very simple. Premet, in \cite[Theorem 2]{Premet} proved the following: if the Gelfand-Kirillov Conjecture holds for a semisimple Lie algebra $\g_{\k, \Sigma}$ in characteristic 0, then its modular version also holds in characterstic p, for p big enough. From a careful analysis of his proof (see Theorem \ref{mainpart1}) we obtain a first-order formula in the language of rings $\mathcal{R}$, denoted by $\phi_{\k, \Sigma}$, such that the Gelfand-Kirillov Conjecture holds for $\g_{\k, \Sigma}$ if and only if $\k \models \phi_{\k, \Sigma}$. Then we show that the formula $\phi_\Sigma$ in Theorem \ref{main1} can be choosen as $\phi_{\A, \Sigma}$, where $\A$ is the field of algebraic numbers.

Let $\operatorname{Tdeg}$ denote the Gelfand-Kirillov transcendence degree of a $\k$ algebra (cf. \cite{Gelfand}, \cite{Zhang}); $trdeg$ will denote the usual transcendence degree of commutative fields. Over a field $\k$ of characteristic 0, $\operatorname{Tdeg} \, \D_{n,l}(\k) = 2n + l$, as computed in \cite[Thm. 2]{Gelfand}. 

\begin{remark}
    One should not confuse the Gelfand-Kirillov transcendence degree with the more used common Gelfand-Kirillov \emph{dimension} (see, e.g., \cite[Chapter 8]{MR}). They don`t coincide: it was shown by L. Makar-Limanov \cite{Makar} that the Weyl fields contain as subalgebras noncommutative free algebras, and hence their Gelfand-Kirillov dimension is infinite, while we just saw that their Gelfand-Kirillov transcendence degree is finite.
\end{remark}

We need to recall some facts of finite dimensional semisimple Lie algebras and their enveloping algebras. Let us fix a root system $\Sigma$, with a fixed basis of simple roots $\Delta$, $n$ in total, and $N$ positive roots in total; the Weyl group will be denoted by $W$. Consider the Lie ring $\g_{\Z, \Sigma}$ obtained by Chevalley basis in the corresponding complex Lie algebra (cf. \cite[25.2]{Humphreys}.)

For every algebraically closed field of characteristic 0 $\k$, we have $\mathfrak{g}_{\k, \Sigma} = \g_{\Z, \Sigma} \otimes_\Z \k$, and similarly $U(\g_{\k,\Sigma}) = U(\g_{\Z, \Sigma}) \otimes_\Z \k$ (cf. \cite[25.5]{Humphreys}, \cite[Chap 1, 2.9]{Bourbaki}). It is also clear that if $\k \subset \k'$, we have $U(\g_{\Z,\Sigma}) \subset U(\g_{\k, \Sigma}) \subset U(\g_{\k', \Sigma})$.

\vspace{0.2cm}

I will denote by $\mathcal{K}=\{ U(\g_{\k, \Sigma})| char \, \k =0, \bar{\k}=\k \}$, $U_\Z= U(\g_{\Z, \Sigma})$. Strictly speaking, $\mathcal{K}$ is a proper class and not a set, but we don't need to worry about this small foundational issue.

The following proposition is well known (e.g., \cite[Chapter 5]{Jacobson}.

\begin{proposition}\label{2}
Every algebra in $\mathcal{K}$ and also $U_\Z$ are Noetherian (hence Ore) domains.
\end{proposition}
    


\begin{proposition}\label{4-5-6}
    Given the algebra $U(\g_{\k,\Sigma})$, we have:

    \begin{enumerate}

\item $Z(\operatorname{Frac} \, U(\g_{\k, \Sigma})) \simeq \operatorname{Frac} \, Z(U(\g_{\k, \Sigma}))$

\item $\operatorname{Tdeg} \, \operatorname{Frac} \, U(\g_{\k, \Sigma}) = \operatorname{dim}_\k \, \g_{\k, \Sigma} $

\item $Z(U(\g_{\k,\Sigma}))$ is a polynomial algebra $\k[x_1, \ldots, x_K]$, where $x_1, \ldots, x_K \in U_\Z$.
        
    \end{enumerate}
\end{proposition}
\begin{proof}
    Item 1 is \cite[4.3.2]{Dixmier}. Item 2 is \cite[Theorem 1.1]{Zhang}. Item 3 is \cite[7.4.6]{Dixmier}.
\end{proof}

The following lemma is obvious, but crucial in what follows.

\begin{lemma}\label{lemma0}
Let $\k$ be an algebraically closed field of zero characteristic, and $U(\g_{\k, \Sigma}) \in \mathcal{K}$. If $\operatorname{Frac} \, U(\g_{\k, \Sigma}) \cong \mathbb{D}_{n,l}(\k)$, then $\operatorname{Frac} \, Z(U(\g_{\k, \Sigma})) \cong \k(x_1,\ldots,x_l)$.
\begin{proof}
Obvious, since $Z( \mathbb{D}_{n,l}(\k)) = \k(x_1,\ldots, x_l)$, cf. \cite[Thm. 2]{Gelfand}, and we have the isomorphism in Proposition \ref{4-5-6}(1)
\end{proof}
\end{lemma}

The following Proposition is a broad generalization of \cite[Lemma 1.2.3]{Bois}, of independent interest.

\begin{proposition}\label{bois}
Let $A$ be an Ore domain over $\k$ of arbitrary characteristic, such that $\operatorname{Tdeg} \, A = 2n+l$.  Suppose we have elements $x_1, \ldots, x_n, y_1, \ldots, y_n, z_1, \ldots, z_l$ of $\operatorname{Frac} \, A$ such that $[x_i, y_j] = \delta_{i,j}$ and all other commutators between these elements vanish. Let $B$ be the subalgebra of $\operatorname{Frac} \, A$ generated by these elements. If $\operatorname{Frac} \, B = \operatorname{Frac} \, A$ then $B \cong A_{n,l}(\k)$ and $\operatorname{Frac} \, A \cong  \D_{n,l}(\k)$.
\end{proposition}
\begin{proof}
By the presentation of the Weyl algebra by generators and relations, we have a surjective homomorphism  $\theta: A_{n,l}(\k) \rightarrow B$ with kernel $I$. Denoting the Gelfand-Kirillov dimension by $\operatorname{GKdim}$, we have by \cite[Theorem 1.1, Proposition 3.1]{Zhang}, $2n+l=\operatorname{GKdim} \, A_{n,l}(\k) \geq \operatorname{GKdim}  \, B \geq \operatorname{Tdeg}  \, B \geq \operatorname{Tdeg} \, \operatorname{Frac} \, B = 2n+l$. Suppose $I \neq \{ 0 \}$. Then implies $\operatorname{GKdim} \, B < \operatorname{GKdim}  \, A_{n,l}(\k)$ (\cite[Corollary 8.3.6]{MR}), which is absurd, $\theta$ is an isomorphism. Hence the claim follows.
\end{proof}

\begin{proposition}\label{charac}
Let $U(\g_{\k, \Sigma})$ be an algebra in $\mathcal{K}$,  $x_1, \ldots, x_n$ be a basis  $\g_{\k, \Sigma}$ contained in $\g_{\Z, \Sigma}$. Then $\operatorname{Frac} \, U(\g_{\k, \Sigma}) \cong \D_{m,l}(\k)$ if and only if we have elements $w_1, \ldots, w_m, w_{m+1}, \ldots,  w_{2m}$  in $\operatorname{Frac} \, U(\g_{\k, \Sigma})$, such that $[w_i,w_j]=[w_{m+i},w_{m+j}]=0, [w_i,w_{m+j}] = \delta_{ij}, i,j=1,\ldots, m$, and polynomials $p_i, q_i$ in $w$'s and coefficients in $Z(U(\g_{\k, \Sigma}))$ such that $q_i x_i = p_i$, $i=1, \ldots, n$.
\begin{proof}
Necessity is clear; the only point worth a remark is about the polynomials $p_i, q_i$. They exist because the algebra generated inside $\operatorname{Frac} \, U(\g_{\k, \Sigma})$ by $w_1, \ldots, w_m,$ $w_{m+1}, \ldots,  w_{2m}$ has as basis monomials on these variables with non-negative integer exponents. Sufficiency follows from Proposition \ref{bois}.
\end{proof}
\end{proposition}

The next result is a kind of ``Nullstellensatz''; its proof follows closely \cite[Theorem 2.1]{Premet}.

\begin{theorem}\label{mainpart1}
Let $U(\g_{\k,\Sigma}) \in \mathcal{K}$. Then, for $N>>0$, there is a locally closed subvariety $X^\k$ of $\k^N$, defined by polynomials with integer coefficients, such that the Gelfand-Kirillov Conjecture holds for $U(\g_{\k,\Sigma})$ if and only if $X^\k \neq \varnothing$.
\begin{proof}
Suppose $\operatorname{Frac} \, U(\g_{\k,\Sigma}) \cong \D_{m,l}(\k)$. By Proposition \ref{charac} we can find elements $x_1, \ldots, x_n  \in \g_{\Z, \Sigma}$ which form a basis of $\g_{\k, \Sigma}$, $w_1, \ldots, w_m, w_{m+1}, \ldots,  w_{2m}$ in $Frac \, U(\g_{\k,\Sigma})$, such that: 
\[ (\dagger) \, [w_i,w_j]=[w_{m+i},w_{m+j}]=0, [w_i,w_{m+j}] = \delta_{ij}, i,j=1,\ldots, m, \] and polynomials $p_i, q_i$ in $w$'s and coefficients in $Z(A_\k)$ such that \[ (\ddagger) \, q_i x_i = p_i, i=1, \ldots, n. \]

 Call $Z(U(\g_{\k,\Sigma}))= \k [\phi_1, \ldots \phi_l]$, $\phi_i \in U(\g_{\Z, \Sigma})$.

 Recall the PBW-filtration of $U(\g_{\k,\Sigma})$ $\mathcal{F}_\k=\{ F_{i,\k} \}_{i \geq 0}$ from Proposition \ref{2}.

Using the fact that $U(\g_{\k,\Sigma})$ is an Ore domain, we can find $a_i, b_j \in F_{d(i), \k}, i=1, \ldots, 2m$, for a certain $d(i) \in \mathbb{N}$, such that $w_i=b_i^{-1}a_i$; and similarly $a_{i,j}, b_{i,j} \in F_{d(i,j), \k}, i,j=1,\ldots,n$, such that

\begin{align} b_{i,j} a_i = a_{i,j} b_j , i,j=1,\ldots,2m.
\end{align}

We can then introduce $c_{i,j}, d_{i,j} \in F_{d(i,j), \k}$ such that

\begin{align}
c_{i,j} b_{i,j} b_i = d_{i,j} b_{j,i} b_j, i,j=1,\ldots, 2m; 
\end{align}

and hence, because of $(\dagger)$,

\begin{align}
c_{i,j} a_{i,j} a_j = d_{i,j} a_{j,i} a_i, i,j=1, \ldots, m \text{ or } i,j=m,\ldots,2m;
\end{align}

\begin{align}
  c_{i,m+j} a_{i,m+j} a_{m+j} = \delta_{i,j}  c_{i,m+j} b_{i,m+j} b_{i}  + d_{i,m+j} a_{m+j, i} a_{i}, i,j=1,\ldots,m.
\end{align}

When $\mathsf{m} \geq 3$ and we have an $\mathsf{m}$-uple $\mathfrak{i}=(i(1),\ldots,i(\mathsf{m}))$ with $1 \leq i(1) \leq \ldots \leq i(\mathsf{m}) \leq 2m$,
we can inductively find elements $a_{i(1)\ldots i(k)}, b_{i(1) \ldots i(k)}$ in $F_{d(\mathbf{i}), \k}^\k$, where $3 \leq k \leq \mathsf{m}$ and such that

\begin{align}
 b_{i(1) \ldots i(k)}a_{i(1)\ldots i(k-1)} a_{i(k-1)}=a_{i(1)\ldots i(k)}b_{i(k)}.
\end{align}

Write $b_\mathfrak{i}= \prod_{k=1}^{\mathsf{m}} b_{i(1) \ldots i(\mathsf{m}-k+1)}$ and $a_\mathfrak{i}=a_{i(1)\ldots i(\mathsf{m})}a_{i(\mathsf{m})}$.  We can consider the tuples $\mathfrak{i}$ as above such that, calling $M = max \{ deg \, p_i, deg \, q_i| i=1,\ldots,n \}$, $\sum_{\ell=1}^{\mathsf{m}} i(\ell) \leq M$. Call them $\{ \mathfrak{i}(1), \ldots, \mathfrak{i}(r) \}$. We have $p_k= \sum_{j=1}^r \lambda_{j,k} b_{\mathfrak{i}(j)}^{-1} a_{\mathfrak{i}(j)}$ and $q_k=\sum_{j=1}^r \mu_{j,k} b_{\mathfrak{i}(j)}^{-1} a_{\mathfrak{i}(j)}$.

We can write \[ \lambda_{j,k} = \sum \lambda_{j,k}(n_1, \ldots, n_l) \phi_1^{n_1} \ldots \phi_l^{n_l}, \]
\[ \mu_{j,k} = \sum \mu_{j,k}(n_1, \ldots, n_l) \phi_1^{n_1} \ldots \phi_l^{n_l}, \]

where summation is over finitely many $l$-tuples $(n_1, \ldots, n_l) \in \mathbb{N}^l$ and $ \lambda_{j,k}(n_1, \ldots, n_l),$ $ \mu_{j,k}(n_1, \ldots, n_l) \in \k$.

There exists $0 \neq  e_{\mathfrak{i}(j);k}, f_{\mathfrak{i}(j);k} \in F_{d(\mathfrak{i}(j);k), \k}$ with

\begin{align}
 a_{\mathfrak{i}(j)} x_k  f_{\mathfrak{i}(j);k} = b_{\mathfrak{i}(j)} e_{\mathfrak{i}(j);k}, j=1, \ldots, r, k=1, \ldots n,
\end{align}

and $p_k = q_k x_k$ implies

\[  \, \sum_{j=1}^r \lambda_{j,k}  b_{\mathfrak{i}(j)}^{-1} a_{\mathfrak{i}(j)} = \sum_{j=1}^r \mu_{j,k} e_{\mathfrak{i}(j);k} f_{\mathfrak{i}(j);k}^{-1}, k=1, \ldots, n. \]

Let us now introduce $a_{\mathfrak{i}(j)}(0)=a_{\mathfrak{i}(j)}, b_{\mathfrak{i}(j)}(0)=b_{\mathfrak{i}(j)}$; $ e_{\mathfrak{i}(j);k}(0)= e_{\mathfrak{i}(j);k};  f_{\mathfrak{i}(j);k}(0)= f_{\mathfrak{i}(j);k}$; and inductively, for each $0<s<j \leq r$, $a_{\mathfrak{i}(j)}(s), b_{\mathfrak{i}(j)}(s)$, $ e_{\mathfrak{i}(j);k}(s), f_{\mathfrak{i}(j);k}(s)$, all in $F_{d(\mathfrak{i}(j),k,s), \k}$ such that the following holds:

\begin{align} b_{\mathfrak{i}(j)}(s)b_{\mathfrak{i}(s)}(s-1)=a_{\mathfrak{i}(j)}(s)b_{\mathfrak{i}(j)}(s-1),
\end{align}

\begin{align} f_{\mathfrak{i}(j);k}(s-1) e_{\mathfrak{i}(j);k}(s) = f_{\mathfrak{i}(s);k}(s-1) f_{\mathfrak{i}(j);k}(s)
\end{align}

Finally, the following algebraic relation holds:

\begin{align}
  (\sum_{j=1}^r \lambda{j,k} \prod_{\ell=1}^{r-j} b_{\mathfrak{i}(r-\ell+1)}(r-\ell) \prod_{\ell=1}^j a_{\mathfrak{i}(j-\ell+1)}(j-\ell)) \prod_{\ell=1}^r f_{\mathfrak{i}(\ell);k}(\ell-1)
\end{align}

\[ = (\prod_{\ell=1}^r  b_{\mathfrak{i}(r-\ell+1)}(r-\ell)  \times \]
\[ (\sum_{j=1}^r \mu_{j,k} \prod_{\ell=1}^j  e_{\mathfrak{i}(\ell);k}(\ell-1) \prod_{\ell=j+1}^r  f_{\mathfrak{i}(\ell);k}(\ell-1)). \]

It follows from the above that we have elements \[ a_i, b_i, a_{i,j}, b_{i,j}, c_{i,j}, d_{i,j}, a_{i(1)\ldots i(s)}, b_{i(i) \ldots i(s)}, a_{\mathbf{i}(j)}(\ell),  b_{\mathbf{i}(j)}(\ell), e_{\mathbf{i}(j);k}(\ell), f_{\mathbf{i}(j);k}(\ell),  \]
where $i,j,k,s, \mathbf{i}(j), \ell$ range through a finite set of indices and belong to $F_{M, \k}^\k$ for $M >>0$, as well two finite collections of elements of $\k$, $\lambda_{j,k}(a_1,\ldots,a_l), \mu_{j,k}(a_1,\ldots,a_l)$, linked be algebraic equations
\[  (1) \,  (2) \,  (3) \,  (4) \, (5) \, (6) \,  (7) \, (8) \, (9). \]
This data can be parametrized by a locally closed subset defined by the simultaneous vanishing of polynomials $\{ f_1, \ldots, f_r \}$ and non-vanishing of some $g \in \{ g_1, \ldots, g_s \}$ in $\k^N$, a sufficiently big affine space; and all these polynomials have integer coefficients. Call this variety $X^\k$. Then the Gelfand-Kirillov Conjecture holding for $U(\g_{\k,\Sigma})$ implies $X^\k \neq \varnothing$. This process is reversible --- hence we have an equivalence of $X^\k \neq \varnothing$ and the Gelfand-Kirillov Conjecture for $U(\g_{\k, \Sigma})$. 
\end{proof}
\end{theorem}

\begin{theorem}
 In the notation of the previous theorem, let $\psi_{\k, \Sigma}:=(\bigwedge f_i = 0)\wedge (\bigvee g_j \neq 0)$, and denote by $\phi_{\k, \Sigma}$ its existential closure. We have that the Gelfand-Kirillov Conjecture holds if and only if $\k \models \phi_{\k,\Sigma}$.
\end{theorem}
\begin{proof}
    $X^\k \neq \varnothing$ if and only if there exists an element $\lambda \in \k$ with $(\bigwedge f_i(\lambda) = 0)\wedge (\bigvee g_j(\lambda) \neq 0)$. This is clearly equivalent to $\k \models \phi_{\k, \Sigma}$
\end{proof}

We shall see that $\phi_{\A, \Sigma}$ is the formula $\phi_\Sigma$ for Theorem \ref{main1}

\begin{corol} \label{mainpart2}
If the Gelfand-Kirillov Conjecture holds for one $U(\g_{\Z, \Sigma}) \in \mathcal{K}$, $\k$ algebraically closed of zero characteristic, then it holds for $U(\g_{\A, \Sigma})$.
\end{corol}
\begin{proof}
By the above Theorem (the notation which we use here), $X^\k \neq \varnothing$. Hence $X^\k(\A) \neq \varnothing$, because $X$ is defined by polynomials with integer coefficients. The process of Theorem \ref{mainpart1} is reversible, and relation (9) therein refers to the ground field; hence, we can assume that elements in $(\dagger), \, (\ddagger)$ belongs to $\operatorname{Frac} \, U(\g_{\A, \Sigma})$. Since the elements $x_1, \ldots, x_n$ also belong to $\mathfrak{g}_{\Z, \Sigma}$ and, by definition, are a basis of this vector space, then we can  apply Proposition $\ref{charac}$ to conclude that Gelfand-Kirillov Conjecture holds for  $U(\g_{\A, \Sigma})$.
\end{proof}

\begin{proposition} \label{mainpart3}
If the Gelfand-Kirillov Conjecture holds for $U(\g_{\A, \Sigma})$, then it holds for all $U(\g_{\k, \Sigma})$, $\k$ any algebraically closed field of zero characteristic.
\begin{proof}
If the Gelfand-Kirillov Conjecture holds for $U(\g_{\A, \Sigma})$ and we can find elements $x's \in U(\g_{\A, \Sigma})$ and $w's \in \operatorname{Frac} \, U(\g_{\A, \Sigma})$ as in Proposition \ref{charac}, then these elements lie, respectively, also in $U(\g_{\k, \Sigma})$ and $\operatorname{Frac} \, U(\g_{\k, \Sigma})$. Then by applying Proposition \ref{charac} in the other direction, we conclude that Gelfand-Kirillov Conjecture also holds for $U(\g_{\k, \Sigma})$.
\end{proof}
\end{proposition}

\begin{theorem}\label{mainpart4}
 The Gelfand-Kirillov Conjecture holds for $U(\g_{\A, \Sigma})$ if and only if $ACF_0 \vdash \phi_\Sigma$, $\phi_\Sigma$ a formula in the language of rings $\mathcal{R}$; and $\phi_\Sigma$ is the existential closure of a boolean combination of atomic formulas.
\begin{proof}
As in the proof Theorem \ref{mainpart1}, there is a locally closed subvariety $X^\A$ of a certain $\A^N$ such that Gelfand-Kirillov Conjecture holds if and only if $X^\A \neq \varnothing$. Recall how $X^\A$ is defined: the simultaneous vanishing of polynomials $\{ f_1, \ldots, f_r \}$ and non-vanishing of some $g \in \{ g_1, \ldots, g_s \}$ in $\A^N$. Let $\psi_\Sigma=(\bigwedge f_i = 0)\wedge (\bigvee g_j \neq 0)$, and denote by $\phi$ the existential closure of $\psi_\Sigma$. Clearly $X^\A \neq \varnothing$ if and only if $\A \vDash \phi$. Since the theory of algebraically closed fields of a fixed characteristic is complete (\cite{Hodges}), this is the case if and only if $ACF_0 \vdash \phi$.
\end{proof}
\end{theorem}

    We are now ready to prove Theorem \ref{main1}
\begin{proof}
(1) implies (2) by Corollary \ref{mainpart2} and Proposition \ref{mainpart3}. (2) implies (1) trivially. Let $\phi_\Sigma$ be as in Theorem \ref{mainpart4}. Then Proposition \ref{mainpart3} and Theorem \ref{mainpart4} show that (2) and (3) are equivalent.
\end{proof}

\section*{Appendix}

We will use results from a previous version of this preprint \cite{MS}. In that version, we defined a notion of $\mathbb{Z}$-compatible family of algebras, which are defined over $\Z$, and satisfy many hypothesis. In that setting, we proved a very general and abstract version of Theorem \ref{main1} (\cite[Theorem 2.2]{MS}). Under closer scrutiny, it seems that many of the imposed conditions for such family of algebras defined over $\Z$ are not necessary. However, the results of that version continue to hold, and we repeat them in this appendix, with some new results also.

The first result we want to recall is  the following (we refer to \cite{SchwarzPan} for a detailed discussion of all notions involved).

\begin{theorem}
    Let $G$ be a finite group of linear automorphisms of the polynomial algebra in $n$ variables, extended naturally to a group of automorphisms of the Weyl algebra $A_n(\k)$, where $\k$ is an algebraically closed field of zero characteristic. Then there exists a formula $\phi_G$, which is a boolean combination of atomic formulas, such that the following are equivalent:

    \begin{enumerate}
        \item Noncommutative Noether's problem holds for $A_n(\k)^G$ some some algebraically closed field of zero characteristic $\k$
        \item Noncommutative Noether's problem holdes for $A_n(\k)$ for \emph{all} algebraically closed fields of zero characteristic.
        \item
        $AFC_0 \vdash \phi_G$
    \end{enumerate}
\end{theorem}

\begin{corollary}
    Let $\k$ be any field of characteristic $0$, not necessarily algebraically closed. Let $\bar{\k}$ denotes its algebraic closure. Let $G$ be a finite group of linear automorphims of $A_n(\k)$ such that the noncommutative Noether's problem has a negative solution for $A_n(\bar{\k})^G$, then it has a negative solution for $A_n(\k)^G$.
\end{corollary}
\begin{proof}
    By the previous theorem, the \emph{negation} of noncommutative Noether's problem for $A_n(\bar{\k})^G$ is equivalent to the validity of a $\prod_1$-formula, and those are preserved by substructures.
\end{proof}

By the same idea, using the fact that any field has an algebraic closure, we can slightly improve Premet's result about the Gelfand-Kirillov Conjecture \cite{Premet}

\begin{corollary}
    If $\Sigma$ is an irreducible root system of type $B,D,E,F$ and $\k$ is any splitting field (of zero characteristic) given a choice of Cartan subalgebra, then the Gelfand-Kirillov Conjecture is false.
\end{corollary}

We notice that the Weyl algebra definition by generators and relations makes perfect sense for a field of prime characteristic $\mathbb{F}$, and $A_n(\mathbb{F})$ will be again a Noetherian domain, with skew field of fractions $\mathbb{D}_n(\mathbb{F})$. The difference is that now $A_n(\mathbb{F})$ is not simple, but an Azumaya algebra over its center, and $\mathbb{D}_n(\mathbb{F})$ is a finite dimensional central simple algebra by Posner's Theorem  \cite[Chapter 13]{McConnell} \cite{BMR} \cite{Bois} \cite{Revoy}.

The Weyl algebra in prime characteristic is different from the usual ring of differential operators on the affine space, due to Grothendieck (see, e.g., \cite{Smith}). However, we can recover it if we consider rings of \emph{crystalline differential operators}, introduced in \cite{BMR}, and whose definition we recall. The difference from this ring and the full ring of differential operators in Grothendieck's sense is the removal of divided powers.

\begin{definition}
Let $\mathbb{F}$ be an algebraically closed field of prime characteristic, $X$ an smooth affine variety. $\mathcal{D}_c(X)$, the ring of crystalline differential operators on $X$, is generated by $\mathcal{O}(X)$ and $Der_\F \mathcal{O}(X)$ subject to the relations

\[ f.\partial=f\partial, \, \partial.f - f.\partial=\partial(f), \]

\[ \partial.\partial'-\partial'.\partial=[\partial,\partial'], f \in \mathcal{O}(X), \partial,\partial' \in Der_\F \mathcal{O}(X).\]

\end{definition}

With this definition, $A_n(\F) \simeq \mathcal{D}_c(\F^n)$. The noncommutative Noether's problem in prime characteristic was studied in \cite{SchwarzPan}, and the Gelfand-Kirillov Conjecture in prime characteristic was studied in \cite{Bois} and \cite{Premet}. In both cases, the statement of the problem is the same is in the 0 characteristic case. Using \cite[Theorem 2.14]{MS}, we have a new proof of \cite[Theorem 2.1]{Premet} and a new result:

\begin{theorem}
    \begin{enumerate}
    \item
        If the Gelfand-Kirillov Conjecture holds for a semisimple Lie algebra over any algebraically closed field $\k$ with zero characteristic with root system $\Sigma$, then its modular version holds for all algebraically closed fields $\mathbb{F}$ with characteristic big enough.

        \item Let $G$ be a finite group of permutations acting naturally on $A_n(\k)$, $\k$ an arbitrary algebraically closed field of zero characteristic. If the noncommutative Noether's problem has a positive solution, then so has its modular version for the same action of $G$, for any algebraically closed field $\F$ of characteristic big enough.
    \end{enumerate}
\end{theorem}

Finally, we move to our last result. In \cite{Tikaradze}, it was shown that if $G$ is any permutation group acting naturally by algebra automorphisms of $A_n(\C)$, then if noncommutative Noether's problem has a positive solution, $\F(x_1, \ldots, x_n)^G$ is stably-rational for all algebraically closed fields $\F$ with prime characteristic big enough. This is a partial converse of the main result of \cite{FS}, that shows that a positive solution to Noether's problem implies a solution to its noncommutative analogue.

We prove a stronger statement, that, in particular, implies the result in \cite{Tikaradze}.

Let $A$ be a commutative unital affine $\Z$ algebra, integral domain, such that the map $\operatorname{Spec} \, A \rightarrow \operatorname{Spec} \Z$ is ètale. For any algebraically closed field $\k$, consider the base change $A_\k: A \otimes_\Z \k$. Since being ètale is stable under base change, $A_\k$ is the algebra of regular functions of an smooth irreducible variety $X_\k$. We have $\mathcal{D}(X_\k)= \Delta(A) \otimes_\Z \k$, where $\Delta(A)$ is the derivation ring of $A$ (\cite[15.1]{McConnell}), and $\mathcal{D}(\cdot)$ is the usual ring of differential operators, if $char \, \k=0$; if $char \, \k=p>0$, then $\Delta(A) \otimes_\Z \k=\mathcal{D}_c(X_\k)$. Our last result, that also can be obtained using the theory of $\Z$-compatible family of algebras, is as follows:

\begin{theorem}
    In the situation as above, suppose $\operatorname{Frac} \mathcal{D}(X_k) \simeq \mathbb{D}_n(\k)$ for some algebraically closed field of zero characteristic. Then $\operatorname{Frac} \mathcal{D}_c(X_\k) \simeq \mathbb{D}_n(\k)$ for all algebraically closed fields $\k$ of prime characteristic big enough.
\end{theorem}

\end{document}